\documentclass[11pt]{article}
\usepackage{amsmath,amssymb,amscd}
\newtheorem{theorem}{Theorem}
\newtheorem{prop}{Proposition}
 
\newtheorem{lem}{Lemma}
\newtheorem{rem}{Remark}{{\rm}}

\newcommand{\benum}{\begin{enumerate}}
\newcommand{\ennum}{\end{enumerate}}

\newcommand{\bs}{\begin{small}}
\newcommand{\es}{\end{small}}

\title{On a perfect isometry between principal $p$-blocks of finite groups with cyclic $p$-hyperfocal subgroups}

\date{}

\author{Hiroshi HORIMOTO and Atumi WATANABE}

\begin{document}

\pagestyle{myheadings}

\maketitle

\begin{abstract}

Let $G$ be a finite group with a Sylow $p$-subgroup $P$. 
We prove that the principal $p$-blocks of $G$ and $N_G(P)$ are perfectly isometric under the assumption $G$ has a  cyclic $p$-hyperfocal subgroup.
\end{abstract}

\section{Introduction and preliminaries}
Let $G$ be a finite group and $p$ be a prime number.  
We denote by $O^{p}(G)$ the  subgroup of $G$ 
generated by the set $G_{p'}$ of $p$-regular elements of $G$. 
Let $P$ be a Sylow $p$-subgroup of $G$ and set 
\[ \tilde{P}  = \langle \ [\ T, \ O^{p}(N_{G}(T)) \ ] \ | \ T \leq P \ \rangle. \]
We will call $\tilde{P}$ a $p$-hyperfocal subgroup of $G$.  
By Puig \cite{Pu5} and  \cite{Pu6}, 
\[  \tilde{P} = P \cap O^{p}(G) \]
(see \cite{BCGLO}, Lemma 2.2 or  \cite{Craven}, Theorem 1.33 for proofs).  
Denoting by $b(G)$ the principal block of $G$, 
 $\tilde{P}$ is a hyperfocal subgroup of $b(G)$ in the sense of Puig 
(\cite{Pu5} and \cite{Pu6}). In this paper we prove the following theorem. This study is motivated by Rouquier \cite{Rou2}, A.2. 

\begin{theorem} If $\tilde{P}$ is cyclic, then  the principal blocks $b(G)$,  $b(N_{G}(\tilde{P}))$ and $b(N_{G}(P))$ are perfectly 
 isometric in the sense of Brou\'{e} \cite {Broue}. 
\end{theorem}
 By \cite{sasaki}, if $p$ is odd and $P$ is metacyclic, then $\tilde{P}$ is cyclic.  If $p = 3$ and $G = SL(2, 2^{3^n}) \rtimes C_3$, then $P$ is isomorphic to the metacylic $3$-group $M_{(n + 1 )+ 1}(3)$ (cf. \cite{HKK}, (4.3)).

Now the above theorem is a generalization of \cite{JWata}, Proposition 6. If $\tilde{P} = 1$, then $G$ is $p$-nilpotent, hence the above is clear.  
We will determine the ordinary irreducible characters in $b(G)$  (Propositions 5 in $\S 2$ below), and the generalized decomposition numbers associated with $b(G)$ (Propositions 6 and 7 in $\S 3$ below), and then we obtain Theorem 1 from Theorem 2 below. 

We refer \cite{NT}, \cite{T} and \cite{Gorenstein} for the notation and terminology. Let $H$ be a subgroup of $G$. For $x \in G$, we denote by $x^H$ the $H$-conjugacy class containing $x$. For a character $\chi$ of $G$, the restriction of $\chi$ to $H$ is denoted by $\chi \!\downarrow_{H}$. For a character $\zeta$ of $H$, the induced character $\zeta$ to $G$ is denoted by $\zeta\!\uparrow_{H}^{G}$. For characters $\chi$, $\chi'$ of $G$, the innner product of $\chi$ and $\chi'$ is denoted by $(\chi, \chi')$. For a complex number $\alpha$ we denote by $\bar{\alpha}$ the complex conjugate, and for a character ${\chi}$ of $G$ we denote by $\bar{\chi}$ the contragredient character. 

 Let $({\cal K}, {\cal O}, F)$ be a $p$-modular system such that ${\cal K}$ contains the field ${\bf Q}(\sqrt[|G|]{1})$.  For a ($p$-)block $b$ of $G$, we denote by Irr$(b)$ the set of ordinary irreducible characters belonging to $b$ and IBr$(b)$ the set of irreducible Brauer characters belonging to $b$. Here $b$ denotes a block idempotent of ${\cal O}G$. We set $k(b) = |{\rm Irr}(b)|$, the cardinality of Irr$(b)$, and $l(b) = |{\rm IBr}(b)|$. Also we set ${\cal R}_{{\cal K}}(G, b) = \sum_{\chi \in {\rm Irr}(b)}\textrm{{\boldmath $Z$}}  \chi$ and we denote by $b(G)$ the principal block of $G$. We denote by $1_{G}$ the trivial irreducible character and the trivial Brauer character of $G$.

Let $(\pi, {\bf b})$ be a $b$-Brauer element, that is, $\bf b$ is a block of $C_{G}(\pi)$ associated with $b$. For $\chi \in {\rm Irr}(b)$, the class function $\chi^{(\pi, {\bf b})}$ of $G$ is defined as follows:   $\chi^{(\pi, {\bf b})}$ vanishes outside of the $p$-section of $G$ containing $\pi$ and $\chi^{(\pi, {\bf b})}(\pi \rho) = \chi(\pi\rho {\bf b})$ for $\rho \in C_{G}(\pi)_{p'}$. Let $Bs({\bf b}) = \{ \varphi^{(\pi)}_j \ | 1 \leq j \leq l({\bf b}) \}$ be a basic set for ${\bf b}$ in the sense of Brauer \cite{Brauer}, that is,  $Bs({\bf b})$ is a 
$\textrm{{\boldmath $Z$}}$-basis of the $\textrm{{\boldmath $Z$}}$-module 
\[\bigoplus_{\varphi \in {\rm IBr}({\bf b})} \textrm{{\boldmath $Z$}}  \varphi. \]
Then there exist ${d}_{\chi {\varphi}_j^{(\pi)}}^{\pi} \in {\cal O}$, what we call the generalized decomposition numbers with respect to $Bs({\bf b})$, which satisfy 
 \[ \chi^{(\pi,{\bf  b})}(\pi \rho) = \sum_{j = 1}^{l({\bf b})} 
{d}_{\chi {\varphi}_{j}^{(\pi)}}^{\pi}
{\varphi}_{j}^{(\pi)}(\rho) \ \ (\forall \rho \in C_{G}(\pi)_{p'}). \]
Set 
\[{D}^{(\pi, {\bf b})}  = \Big( {d}_{\chi {\varphi}_{j}^{(\pi)}}^{\pi}\Big)_
{\chi \in {\rm Irr}(b), \varphi_{j}^{(\pi)} \in Bs({\bf b})}. \]
By \cite{NT}, Theorem 5.4.11,  the matrix 
${^t}({D}^{(\pi, {\bf b})})\overline{D^{(\pi, {\bf b})}}$ is similar to the Cartan matrix of ${\bf  b}$ {w.r.t.} the basic set $Bs({\bf b})$, where $\overline{D^{(\pi, {\bf b})}}$ is the complex conjugate matrix of ${D}^{(\pi, {\bf b})}$ (see \cite{Navarro}, Lemma 7.5).    
The following  plays an important role in this paper.  
% Theorem 2 
\begin{theorem} Let $b$ $(resp. \  b')$ be a block of a finite group $G$ $(resp. \ 
G')$ such that $b$ and $b'$ satisfy  {\rm (i) - (iv)} below.   

{\rm (i)} $b$ and $b'$ has a common defect group $P$. Let $(P, b_P)$ $(resp. \ (P, b'_P))$ be  a maximal $b$ $resp. \ b')$-Brauer pair,  

{\rm (ii)} there exists  $\Pi \subseteq P$ such that 
$\{ (\pi, b_{\pi}) \in (P, b_P) \ | \   \pi \in \Pi \}$ $(resp. \ \{ (\pi, b'_{\pi}) \in (P, b'_P) \ |\  \pi \in \Pi \})$ is a set of representatives for the $G$ $(resp. \ G')$-conjugacy classes of $b$ $(resp. \ b')$- Brauer elements, 

{\rm (iii)} $l(b_{\pi}) = l(b'_{\pi})$ for any $\pi \in \Pi$, 

{\rm (iv)} For any $\pi \in \Pi \ \backslash \ \{ 1 \}$, there exist basic set $Bs(b_{\pi}) = \{ \varphi_{j}^{(\pi)} \ | \  j = 1, \ 2, \cdots, l(b_{\pi}) \}$ for $b_{\pi}$ and $Bs(b'_{\pi}) = \{ {\varphi'_j}^{(\pi)} \ | \  j = 1, \ 2, \cdots, l(b'_{\pi}) \}$ for $b'_{\pi}$ such that 
\begin{equation} d_{\chi_i \varphi_{j}^{(\pi)}}^{\pi} = \varepsilon_i  d_{\chi'_{i} \varphi_{j}'^{(\pi)}}^{\pi}, \ \varepsilon_i = \pm 1 \ (1 \leq i \leq k(b), \ 
1 \leq j \leq l(b_{\pi}))
 \end{equation} where {\rm Irr}$(b) = \{ \chi_i \ | \ i = 1, 2, \cdots, k(b)\}$ and {\rm Irr}$(b') = \{ \chi'_i \ | \ i = 1, 2, \cdots, k(b')\}$. Then  
\[ \chi'_i \in R_{{\cal K}}(G', b') \mapsto \varepsilon_{i} \chi_i \in R_{{\cal K}}(G, b) \]
is a perfect isometry from  $b'$ to $b$ in the sense of Brou\'{e} \cite{Broue}. 
\end{theorem}
\begin{rem} By the condition (ii) and (iii),  $k(b) = k(b')$. 
\end{rem}
% Theorem 2 証明
\begin{proof}(Separation condition) Let $a'$ be a $p$-element of $G'$ and let $\Pi'_{a'}$ be the set of elements $\pi$ of $\Pi$ which is $G'$-conjugate to $a'$. For $\pi \in \Pi'_{a'}$, let $c'_{\pi a'}\in G'$ be such that $a'^{c'_{\pi a'}} = (a')^{c'_{\pi a'}} = \pi$. 
 Similarly for a $p$-element $a$ of $G$, we define $\Pi_a$ and $c_{\pi a}$ for $\pi \in \Pi_{a}$. Now for $\rho \in G_{p'}$ and ${\rho'} \in C_{G'}(a')_{p'}$, from (1) and by the assumption and \cite{NT}, Theorem  5.4.11 for basic sets, 
\[ \sum_{i= 1}^{k(b)} \varepsilon_{i} \chi_i (\rho)\overline{{\chi}'_i (a' {\rho}'}) = \sum_{i = 1}^{k(b)} \chi_i(\rho) \left( \sum_{\pi \in \Pi'_{a'}} \sum_{j =1}^{l(b_{\pi})} \ \overline{d^{x}_{\chi_i \varphi_{j}^{(\pi)}}}\overline{{\varphi'}_{j}^{(\pi)} ({\rho}'^{c'_{\pi a'}})}\ \right) \]
\[=  \sum_{\pi \in \Pi'_{a'}} \sum_{j = 1}^{l(b_{\pi})} \left( \sum_{i= 1}^{k(b)} \chi_i(\rho) \overline{d^{\pi}_{\chi_i \varphi_{j}^{(\pi)}}}\ \right) \overline{\varphi'^{(\pi)}_{j} ({\rho'}^{{c'}_{\pi a'}})} = 0.  \]
(see \cite{Navarro}, Lemma 7.5) Similarly for a $p$-singular element $x \in G$ and 
 $y' \in G'_{p'}$, \\ $\displaystyle\sum_{i=1}^{k(b)} \varepsilon_{i}\chi_i(x)\overline{\chi'_{i}(y')} = 0.$ 
% integrality  condition 
 (Integrality condition)  Let $\rho \in G_{p'}$ and ${\rho'}\in G'_{p'}$. 
 By the above, $\displaystyle\sum_{i= 1}^{k(b)} \varepsilon_i \overline{\chi'_i({\rho'})}\chi_i$ vanishes on the $p$-singular elements of $G$. Hence $\displaystyle\sum_{i= 1}^{k(b)}\varepsilon_i \chi_i(\rho)\overline{\chi'_i ({\rho'})}$ is divisible by $|C_{G}(\rho)|$ in  ${\cal O}$ by \cite{NT}, Theorems 3.6.10 and 3.6.13. Similarly $\displaystyle\sum_{i= 1}^{k(b)}\varepsilon_i \chi_i(\rho)\overline{\chi'_i ({\rho'})}$ is divisible by $|C_{G'}(\rho')|$. Let $a$ (resp. $a'$) be a $p$-element of $G$ (resp. $G'$). For $\rho \in C_{G}(a)_{p'}$ and $\rho'  \in C_{G'}(a')_{p'}$,  
\[\sum_{i= 1}^{k(b)}\varepsilon_i \chi_i (a\rho) \overline{\chi'_i(a'{\rho}')} \]
\[ = \sum_{i = 1}^{k(b)}\left( \sum_{ \pi \in \Pi_a } \sum_{j =1}^{l(b_{\pi})}\varepsilon_i d^{\pi}_{\chi_i \varphi_{j}^{(\pi)}} {\varphi}_{j}^{(\pi)} ({\rho}^{c_{\pi a}}) \right)\left( \sum_{\pi' \in \Pi'_{a'}} \sum_{k =1}^{l(b_{\pi'})}\overline{ d^{\pi'}_{\chi'_i  \varphi'^{(\pi')}_{k}}}\overline{{\varphi'}_{k}^{(\pi')} ({\rho}'^{c'_{\pi' a'}})} \right) \]
\[=  \sum_{\pi \in \Pi_a}\sum_{\pi' \in \Pi'_{a'}}\sum_{k =1}^{l(b_{\pi'})}\left( \sum_{j =1}^{l(b_{\pi})}\left( \sum_{i = 1}^{k(b)} \varepsilon_i d^{\pi}_{\chi_i  \varphi_{j}^{(\pi)}}
\overline{d^{\pi'}_{\chi'_i  \varphi'^{(\pi')}_{k}}}\right) \varphi^{(\pi)}_{j}({\rho}^{c_{\pi a}})\right)\overline{{\varphi'}_{k}^{(\pi')} ({\rho}'^{c'_{\pi' a'}})}.  \]
Hence from (1) and  by \cite{NT}, Theorem 5.4.11 for basic sets (see \cite{Navarro}, Lemma 7.5)
\[\sum_{i= 1}^{k(b)}\varepsilon_i \chi_i (a\rho) \overline{\chi'_i(a'{\rho}')}\]
\[ = \sum_{\pi \in \Pi_a \cap \Pi'_{a'}}
\sum_{k = 1}^{l(b_{\pi})}\left( \sum_{j= 1}^{l(b_{\pi})} 
 \left( \sum_{i = 1}^{k(b)} d^{\pi}_{\chi_i \varphi_j^{(\pi)}}
\overline{d^{\pi}_{\chi_i  \varphi_k^{(\pi)}}}\right) \varphi_j^{(\pi)}(\rho^{c_{\pi a}}) \right)\overline{ {\varphi'}^{(\pi)}_{k}(\rho'^{c'_{\pi a'}})}\]
\[= \sum_{\pi \in \Pi_a \cap \Pi'_{a'}}
\sum_{k= 1}^{l(b_{\pi})} \left( \sum_{j = 1}^{l(b_{\pi})} 
c_{\varphi_{j}^{(\pi)}\varphi_{k}^{(\pi)} } \varphi_{j}^{(\pi)}(\rho^{c_{\pi a}})
\right)\overline{\varphi'^{(\pi)}_{k}(\rho'^{c'_{\pi a'}})}\big), \]
here $c_{\varphi_j^{(\pi)} \varphi_k^{(\pi)}}$ is  a Cartan integer {w.r.t.}  the basic set $Bs(b_{\pi})$ in (iv). Thus $\displaystyle\sum_{i= 1}^{k(b)}\varepsilon_i \chi_i (a\rho)\overline{\chi'_i(a'{\rho}')}$ is divisible by $|C_{G}(a\rho)|$ in ${\cal O}$, because  $\displaystyle \sum_{j = 1}^{l(b_{\pi})} 
c_{\varphi_j^{(\pi)} \varphi_k^{(\pi)}} \varphi_j^{(\pi)}(\rho^{c_{\pi a}}) $ is divisible 
by $|C_{G}(\pi \rho^{c_{\pi a}})| = |C_{G}(a\rho)|$ in ${\cal O}$. Similarly $\displaystyle\sum_{i= 1}^{k(b)}\varepsilon_i \chi_i (a\rho)\overline{\chi'_i(a'{\rho}')}$ is by  $|C_{G'}(a'\rho')|$ in ${\cal O}$. This completes the proof. 
$\blacksquare$
\end{proof}

 Let $P$ be a Sylow $p$-subgroup of $G$ and denote by ${\cal F}_{P}(G)$ the Frobenius category of $G$. The objects of ${\cal F}_{P}(G)$ are subgroups of $P$. For $S, \ T \leq P$, \[{\rm mor}_{{\cal F}_{P}(G)}(S, \ T) = \{ \varphi \in {\rm Hom}(S, \ T) \ | \ \varphi = {c_{x}}_{|_S} \ ( \exists  x \in G) \},  \] where $c_x$ is the inner automorphism of $G$ induced by $x$.

% Prop 1

\begin{prop} {\rm (\cite{Okuyama-Watanabe})} If 
  $\tilde{P}$ is abelian, then  ${\cal F}_{P}(G) = {\cal F}_{P}(N_{G}(\tilde{P})).$ 
\end{prop} 
\begin{proof} (Naoki Chigira)
$\tilde{P} = P \cap O^{p}(G)$ is strongly closed in $P$ in the sense of \cite{GLS}, \S 16. Hence the proposition follows from  \cite{GLS}, Proposition 16.20. $\blacksquare$
\end{proof}

% Prop 2 
\begin{prop} {\rm(\cite{Wata2014}, Theorem 3)} If $\tilde{P}$ is cyclic, then 
 ${\cal F}_{P}(G) = {\cal F}_{P}(N_{G}(P)).$  
\end{prop}

\begin{proof} Set $\tilde{K} = O_{p'}(C_{{G}}(\tilde{P})) = O_{p'}(C_{\tilde{G}}(\tilde{P})) = C_{\tilde{G}}(\tilde{P})_{p'}$. Then $C_{\tilde{G}}(\tilde{P}) =  \tilde{K} \times \tilde{P}$, and $\tilde{K}$ is normal in $N_{G}(\tilde{P})$. 
Since $C_{P}(\tilde{P})$ is a Sylow $p$-subgroup of $C_{G}(\tilde{P})$, 
$C_{G}(\tilde{P}) = C_{\tilde{G}}(\tilde{P}) C_{P}(\tilde{P})  = \tilde{K} C_{P}(\tilde{P})$, by Proposition 1,  
%(2)
\begin{equation}
N_{G}(\tilde{P}) = C_{G}(\tilde{P})N_{G}(P) = \tilde{K} N_{G}({P}). 
\end{equation}
This implies 
\[{\cal F}_{P}(N_{G}(\tilde{P})) = {\cal F}_{P}(N_{G}(P)). \]
Note that $[P, K]  \cap P = 1$. This and Proposition 1 complete the proof. $\blacksquare$
 \end{proof} 

%remark 
\begin{rem} The above two propositions are generalized in  \cite{Wata2014}) 
as follows. Let  $b$ be a block (not necessarily the principal block) of $G$ with a maximal $b$-Brauer pair $(P, b_P)$ and a hyperfocal subgroup $\tilde{P}$. Let $b' \in Bl(N_{G}(P))$ be the 
Brauer correspondent of $b$. If $\tilde{P}$ is cyclic, then the Brauer categories ${\cal F}_{(P, b_P)}(G, b)$ and $ {\cal F}_{(P, b_P)}(N_{G}(P),  b')$ are same. 
\end{rem} 

Set $K = O_{p'}(C_{G}(P))$ and let ${T}$ be a $p$-complement of $N_{G}(P)$. Then $K = T \cap C_{G}(P)$. 
From (2), $\tilde{T} : = \tilde{K}T$ is a $p$-complement of $N_{\tilde{G}}(\tilde{P})$. Then 
\begin{equation}
\tilde{T}/\tilde{K} \cong T/K
\end{equation}
because $T \cap C_{G}(\tilde{P}) \subseteq  C_{G}(P)$. 

In the rest of this paper $G$ is a finite group  and $P$ is a Sylow $p$-subgroup of $G$,  $\tilde{G} = O^{p}(G), \tilde{P} = P \cap O^{p}(G)$ and 
$N =  N_{G}(P)$.  Moreover $e = |E|$, the inertial index of $b(G)$. 
 We assume that $\tilde{P}$ is a cyclic group generated by $x$ and  of order $p^{n}$. 
From (3), $e$ is the inertial index of $b(\tilde{G})$ and hence $e$ divides $p-1$.

If $e = 1$, then $G$ is $p$-nilpotent. Hence we may assume $e \neq 1$ and hence $p \neq 2$ in a proof of Theorem 1. For a $p$-element $u \in G$, we denote by $e_{u}$ 
the inertial index of $b(C_{G}(u))$ and we call it the inertial index of $u$. 
 
% 2
\section{Irreducible characters in  $b(G)$}

% Lemma 1
\begin{lem}  Under our assumption, let $u \in P$. Then $P_u = C_{P}(u)$ is a Sylow $p$-subgroup of $C_{G}(u)$ and $e_u = |C_{N}(u) : P_u C_{N}(P_u)|$.  
\end{lem}
\begin{proof} There exists $a \in G$ such that $P \cap C_{G}(u^a)$ is a Sylow $p$-subgroup of $C_{G}(u^a)$. 
 Then $u^a \in P \cap C_{G}(u^a) \subseteq P$. By Proposition 2, there exists $n \in N$ such that $u^a = u^n$. Since 
 $P \cap C_{G}(u)^n$ is a Sylow $p$-subgroup of $C_{G}(u)^n $,  $P_u = P \cap C_{G}(u) = P^{n^{-1}} \cap C_{G}(u) $ is a 
Sylow $p$-subgroup of $C_{G}(u)$. 

By Proposition 2, we have  $N_{G}(P_u) = C_{G}(P_u)(N_{G}(P_u) \cap N)$, hence  \begin{eqnarray*}
C_{G}(u) \cap N_{G}(P_u) &=& C_{G}(P_u)(C_{G}(u) \cap N_{G}(P_u) \cap N)\\
 &=&  C_{G}(P_u)(C_{G}(u) \cap N). 
\end{eqnarray*}
On the other hand we have  $C_{G}(u) \cap C_{G}(P_u)P_u = C_{G}(P_u)P_u$. Therefore 
\begin{eqnarray*}
 e_u &=& |C_{G}(u) \cap N_{G}(P_u) : (C_{G}(u) \cap C_{G}(P_u))P_u|\\
&=& |C_{G}(P_u)(C_{G}(u) \cap N) :  C_{G}(P_u)P_u|\\
&=&|C_{G}(u) \cap N : C_{G}(u) \cap N \cap C_{G}(P_u)P_u| \\
&=& |C_{N}(u) :  N \cap C_{G}(P_u)P_u| \\
&=& |C_{N}(u) : C_{N}(P_u)P_u|. 
\end{eqnarray*}$\blacksquare$ \end{proof}

\bigskip

The following is special case of  \cite{Wata2014} Lemma 4 (i).  
For the self-containedness we state it  with a proof because this is crucial.  
%Lemma 2
\begin{lem} {\rm(\cite{Wata2014} Lemma 4 (i))} Suppose that $\tilde{P} \neq 1$. Let $T$ be a $p$-complement of 
$N$ and let $X$ be a subgroup of $T$ 
containing $O_{p'}(C_{G}(P))$ properly. Then 
$P = \tilde{P}\rtimes C_{P}(X). $ In particular $C_{P}(X) = C_{P}(T)$ and $G = \tilde{G}\rtimes C_{P}(T)$.  
\end{lem} 
\begin{proof}Let $u \in P$. Since $\tilde{P}$ is a $p$-hyperfocal subgroup of $G$, $[u, X] \subseteq [P, X] \subseteq \tilde{P}$. Hence 
$\tilde{P}u$ is $X$-invariant. On the other hand $\tilde{P}$ acts transitively on 
$u\tilde{P}u$ by right  multiplication. By using a lemma of Glauberman (\cite{Isaacs}, 13.8)
 there is an element of $\tilde{P}u$ fixed by $X$. Since $\tilde{P}$ is cyclic, from (2) and (3), $X/O_{p'}(C_{G}(P))$ acts on $\tilde{P} \backslash \{ 1\}$ fixed-point freely, an element of $\tilde{P}u$ fixed by  $X$ is unique. This implies 
$P = \tilde{P}\rtimes C_{P}(X)$. Since $G = \tilde{G}P$, we have $G = \tilde{G} \rtimes C_{P}(T)$. $\blacksquare$ \end{proof}  

%Lemma 3 
\begin{lem} For $u \in P$, $e_{u} = e$ or $e_{u} = 1$. 
If $e_{u} = e$, then $C_{N}(u) = P_u T$ for some $p$-complement of $N$ and $C_{\tilde{P}}(u)$ is a $p$-hyperfocal subgroup of $C_{G}(u)$. If $e_{u} = 1$, then $C_{N}(u) = P_u O_{p'}(C_{G}(P))$. 
\end{lem}
\begin{proof} By Lemma 2, $P_u = C_{P}(u)$ is a Sylow $p$-subgroup of $C_{G}(u)$. By using the Schur-Zassenhaus theorem,  we can see $C_{N}(u) = P_u C_{T}(u)$ and  $C_{N}(P_u) = Z(P_u) C_{T}(P_u)$ for some $p$-complement $T$ of $N$. Hence 
\[e_u = |C_{N}(u) : P_u C_{N}(P_u)| = |C_{T}(u): C_{T}(P_u)|. \]
Now suppose that $e_u \neq 1$. Then $O_{p'}(C_{G}(P)) \subseteq C_{T}(P_u) \neq  C_T(u)$. So by applying Lemma 3 for $X = C_{T}(u)$, $u \in C_{P}(T)$ and hence $C_{N}(u) = P_u T$. (So we have $ u^N = u^P$. ) Since $\langle x^{p^{n-1}} \rangle $ is a normal subgroup of $P$ of order $p$,  $\langle x^{p^{n-1}} \rangle \subseteq C_{P}(u) = P_u$. >From (2) and (3), $C_{T}(x^{p^{n-1}})= O_{p'}(C_{G}(P))$, hence $C_{T}(P_u) = O_{p'}(C_{G}(P))$. Therefore we see  $e_ u = e$. Moreover  $C_{\tilde{P}}(u)$ is a $p$-hyperfocal subgroup of $C_{G}(u)$  because $[C_{T}(u), C_{\tilde{P}}(u)] = C_{\tilde{P}}(u)$. 

Finally suppose $e_u = 1.$ Then $O_{p'}(C_{G}(P)) = C_{T}(u)$ by the argument as in the case $e_u \neq 1$,  and hence $C_{N}(u) = P_{u}O_{p'}(C_{G}(P))$. (Then  we see $|u^{N}| = e |u^{P}|$.) $\blacksquare$ \end{proof}\\

Let $\Pi \subseteq P$ be a set of representatives for the conjugacy classes of $p$-elements of $G$. Let $\eta$ be a $G$-stable generalized character of $P$, that is, if $u$ and $u'$ in $P$ are $G$-conjugate, then $\eta(u) = \eta(u')$. By a theorem of Brauer on generalized characters (\cite{NT}, Theorem 3.4.2) and the second and third main theorem on principal blocks, if $\chi \in {\rm Irr}(b(G))$ then 
\[ \chi * \eta =  \sum_{u \in \Pi }\eta(u)\chi^{(u, b(C_{G}(u)))} \]
is a generalized character belonging to $b(G)$ (c.f. \cite{BP}).

Let $\tilde{{\cal X}}$ be a set of representatives for the $E$-conjugacy classes of non-trivial linear characters of $\tilde{P}$. For each $\mu \in \tilde{{\cal X}}$, we set
\[ \eta_{\mu} = 
\sum_{a \in E}\mu^a = \sum_{a \in N_{\tilde{G}}(\tilde{P})/C_{\tilde{G}}
(\tilde{P})}\mu^a. \]
This is $\tilde{G}$-stable. 

% Lemma 4 
\begin{lem} Suppose that a finite group $G$ has a cyclic Sylow $p$-subgroup $P$. Then the trivial character $1_G$ is not 
an exceptional character.  
\end{lem}
\begin{proof} We use \cite{Dor}, Theorem  68.1 (cf. \cite{D}). Set $e = |N_{G}(P): C_{G}(P)|$ and $m  = \frac{|P| -1 }{e}$. Then $e | (p - 1)$. 
If $m = 1$, then any irreducible character in $b(G)$ is not  exceptional. So we may assume $m \neq 1$. 
 Assume that  $1_{{G}}$ is exceptional.  For some non-trivial linear character $\mu$ of $P$, 
\[1 = 1_G(\pi) = \varepsilon \sum_{x \in N_{G}(P)/C_{G}(P)} {\mu}(\pi^x) \ (\forall \pi \in {P} \backslash \{ 1 \}), \]
where $\varepsilon = \pm 1$. Then we have $\varepsilon m = \sum_{\pi \in {P} \backslash \{ 1 \}} \mu(\pi) = -1$. This is a contradiction. $\blacksquare$ \end{proof} \\  

 % Lemma 5
\begin{lem} {\rm (Dade \cite{D})}
 ${\rm Irr}(b(\tilde{G})) = \{  \tilde{\chi}_1 = 1_{\tilde{G}},\tilde{\chi}_2, 
\cdots, \tilde{\chi}_e, \tilde{\chi}_{\mu} \ ( \mu \in \tilde{{\cal X}})\},$ where 
\begin{equation} \left\{ \begin{array}{l}
\tilde{\chi}_j (u\rho) = \varepsilon_j = \pm 1 \ ( 1 \leq j \leq e),   \\ 
\tilde{\chi}_{\mu}(u \rho) = \varepsilon \eta_{\mu}(u) \ (\varepsilon = \pm 1)\end{array} \right. \end{equation}
for $u (\neq 1) \in \tilde{P}$ and $ \rho \in (C_{\tilde{G}}(u))_{p'}$. Moreover if $P$ is normal in $G$, then $\varepsilon_j = 1$ and $\varepsilon = 1$.  
\end{lem}
\begin{proof} Since $b(\tilde{G})$ has a cyclic defect group $\tilde{P}$, the lemma follows by Lemma 4 and \cite{Dor}, Theorem 68.1.   $\blacksquare$ \end{proof}

% Lemma 6    
\begin{lem} For  $\mu \in \tilde{{\cal X}}$, we have $1_{\tilde{G}} * \eta_{\mu} = (e -1)1_{\tilde{G}} - \displaystyle\sum_{i = 2}^{e}\varepsilon_i \tilde{\chi}_i + \varepsilon \tilde{\chi}_{\mu}. $
\end{lem}
\begin{proof} Let $u \in \tilde{P}\backslash \{ 1\}$ and $\rho \in C_{\tilde{G}}(u)_{p'}$. 
>From (4), we have
\[(1_{\tilde{G}} * \eta_{\mu})(u\rho) = \eta_{\mu}(u), \] 
\[ ((e -1)1_{\tilde{G}} - \sum_{i = 2}^{e}
\varepsilon_i \tilde{\chi}_i + \varepsilon \tilde{\chi}_{\mu})(u\rho) = \eta_{\mu}(u). \]
Therefore 
 $1_{\tilde{G}} * \eta_{\mu} = (e -1)1_{\tilde{G}} - \displaystyle\sum_{i = 2}^{e}\varepsilon_i \tilde{\chi}_i + \varepsilon \tilde{\chi}_{\mu}$ on $p$-singular elements of $\tilde{G}$.

On the other hand for  $\tau \in \tilde{G}_{p'}$, by \cite{NT}, Theorem 5.4.5, the equation 
 (4) and \cite{Dor}, Theorem 68.1, (6), 
\[  0 = \sum_{i =1}^{e}\tilde{\chi}_{i}(u)\tilde{\chi}_{i}(\tau) + \sum_{\mu \in \tilde{{\cal X}}}\tilde{\chi}_{\mu}(u)\tilde{\chi}_{\mu}(\tau)
\]\[= 1 + \sum_{i =  2}^{e}\varepsilon_i \tilde{\chi}_i(\tau) +  
 \big( \varepsilon \sum_{\nu \in \tilde{{\cal X}}}\eta_{\nu}(u)\big)\tilde{\chi}_{\mu}(\tau)   \]
\[ =  1 + \sum_{i =  2}^{e}\varepsilon_i \tilde{\chi}_i(\tau) -  
\varepsilon \tilde{\chi}_{\mu}(\tau).  \]
From this 
\[(1 * \eta_{\mu})(\tau) = e = \big((e-1)1_{\tilde{G}} - \sum_{i = 2}^{e}\varepsilon_i \tilde{\chi}_i (\tau) + \varepsilon\tilde{\chi}_{\mu}\big)(\tau). \]
This completes the proof. $\blacksquare$
\end{proof} \\

We denote by Irr$_{P}(\tilde{P})$ the set of $P$-invariant linear characters of $\tilde{P}$. Note that if $\nu \in {\rm Irr}_{P}(\tilde{P})$ then $\nu^{a} \in {\rm Irr}_{P}(\tilde{P})$ for $a \in N$. Let $\nu \in $Irr$_{P}(\tilde{P})$. Then $\nu$ is trivial on $[\tilde{P}, P]$. Since $P/[\tilde{P}, P] = (\tilde{P}/[\tilde{P}, P]) \times (C_{P}(T)[\tilde{P}, P]/[\tilde{P}, P])$ by Lemma 2, there is a unique extension $\hat{\nu}$ of $\nu$ to $P$ such that $C_{P}(T) \subseteq {\rm Ker}\  \hat{\nu}$, where $T$ is a $p$-complement of $N$. In fact $C_{P}(U) \subseteq {\rm Ker}\  \hat{\nu}$ for any $p$-complement $U$ of $N$. We call $\hat{\nu}$ the $canonical$ $extension$ of $\nu$. For $a \in N$, $({\hat{\nu}})^a = \widehat{\nu^{a}}$. We set
\[ \eta_{\hat{\nu}} = \sum_{a \in E}\hat{\nu}^a. \]
This is  $N$-invariant, and hence $\eta_{\hat{\nu}}$ is $G$-stable by Proposition 2. Of course $\eta_{\hat{\nu}}\downarrow_{\tilde{P}} = \eta_{\nu}$. Also note $\eta_{\hat{\nu}}(u) = e$ if $u \in C_{P}(T)$.

% Prop 3
\begin{prop} Let $\nu \in {\rm Irr}_{P}(\tilde{P})$. Under the above notations, 
\begin{equation}1_{{G}} * \eta_{\hat{\nu}} = (e - 1)1_{{G}} - \sum_{i = 2}^{e}\varepsilon_i {\chi}_i + \varepsilon {\chi}_{\nu}, \ (\forall \nu \in {\rm Irr}_P(\tilde{P}))
\end{equation}
where $\chi_i$(resp. $\chi_{\nu}$) is an extension of $\tilde{\chi}_i$ (resp. $\tilde{\chi}_{\nu})$ to $G$.  
\end{prop}
\begin{proof} Since $b(\tilde{G})$ has a cyclic defect group $\tilde{P}$ and since $C_{\tilde{G}}(u)$ is $p$-nilpotent for any $u \in \tilde{P} \backslash \{ 1 \}$, 
\[(*) \ \ \ (1_{\tilde{G}}^{(1, b(\tilde{G}))}, 1_{\tilde{G}}^{(1, b(\tilde{G}))})  = 1 - \sum_{u}(1_{C_{\tilde{G}}(u)}^{(u, b(C_{\tilde{G}}(u))}, 1_{C_{\tilde{G}}(u)}^{(u, b(C_{\tilde{G}}(u))})\]
\[ = 1 - \frac{p^{n} -1}{e}\frac{1}{p^n} = \frac{(e-1)p^n + 1}{ep^{n}},  \]
where $u$ runs over a set of representatives for the conjugacy classes of the $p$-elements of $\tilde{G}$. Hence we have 
\[ (1_G^{(1, b(G))}, 1_G^{(1, b(G))}) = (1_G^{(1, b(\tilde{G}))}, 1_G^{(1, b(\tilde{G}))})\frac{1}{|G:\tilde{G}|} = \frac{(e-1)p^n + 1}{e|P|}.\] 
Let $\Pi$ be a set of representatives for the $N$-conjugacy classes of $P \backslash \{ 1 \}$. By Proposition 2, $\Pi$ is a set of representatives for the conjugacy  classes of $p$-elements of $G$. Let $\Pi_1 = \{ u \in \Pi \ | \  e_u = 1 \}$ and $\Pi_2 = \Pi \backslash \Pi_1 =  \{ u \in \Pi \ | \ e_u = e \}$. Set $b_u = b(C_{G}(u))$ for $u \in \Pi$ and $U = \bigcup_{u \in \Pi_2} u^P$. By Lemmas 1 and 2, 
\[ \sum_{u \in \Pi_1}(1_{G}^{(u, b_u)}, 1_{G}^{(u, b_u)})  =  \sum_{u \in \Pi_1}|C_{P}(u)|^{-1} \]
\[ = \frac{1}{e|P|} \sum_{u \in \Pi_1} e |P: C_{P}(u)| = \frac{1}{e|P|}\big(|P| - |U| - 1).\]
Hence we have 
\[ (**) \ \ \ \ \ \ \sum_{u \in \Pi_2}(1_{G}^{(u, b_u)}, 1_{G}^{(u, b_u)}) \]
\[ = 1 - (1_{G}^{(1, b(G))}, 1_{G}^{(1, b(G))}) 
-  \sum_{u \in \Pi_1}(1_{G}^{(u, b_u)}, 1_{G}^{(u, b_u)}) \]
\[ =  \frac{(e - 1)|P| - (e - 1)p^n  + |U|}{e|P|}. \]
Since $(\eta_{\hat{\nu}}, 1_{P}) = 0$, 
\[e + \sum_{u \in \Pi_1}\eta_{\hat{\nu}}(u) e |P :C_{P}(u)| + \sum_{u \in \Pi_2}\eta_{\hat{\nu}}(u) |P : C_{P}(u)| = 0,  \]
and hence  
\[ (***) \ \ \ \ \ \ \  \sum_{u \in \Pi_1}\eta_{\hat{\nu}}(u) |C_{P}(u)|^{-1}\]
\[= -|P|^{-1} - \sum_{u \in \Pi_2} |C_{P}(u)|^{-1}   = -|P|^{-1} - |P|^{-1}|U| \]
because $\eta_{\hat{\nu}}(u) = e$ for $u \in \Pi_2$.  From $(*), (**)$ and $(***)$ 
\[ ( 1_G * \eta_{\hat{\nu}}, 1_G)  \]
\[= e(1_{G}^{(1, b(G))}, 1_{G}^{(1, b(G))})  +  \sum_{u \in \Pi_1}\eta_{\hat{\nu}}(u)(1_{G}^{(u, b_u)}, 1_{G}^{(u, b_u)}) +  e\sum_{u \in \Pi_2}(1_{G}^{(u, b_u)}, 1_{G}^{(u, b_u)}) \]
\[= \frac{(e - 1)p^n + 1}{|P|}  + (- |P|^{-1} - |P|^{-1}|U|) + \frac{(e - 1)|P| - (e - 1)p^n + |U|}{{|P|}}\]
\[ = e - 1.   \]
From this  we have 
\[
 (1_G * \eta_{\hat{\nu}}, 1_G * \eta_{\hat{\nu}} ) =  
(1_G * \eta_{\hat{\nu}}\eta_{\hat{\nu}^{-1}}, 
1_G) =  e + (e-1)^2 = e^2 - e + 1.  \]
Moreover, for $\mu, \mu' (\neq) \in {\cal M}$, 
 \[ (1_G * \eta_{\hat{\mu'}}, 1_G * \eta_{\hat{\mu}} ) =  (1_G * \eta_{\hat{\mu'}}\eta_{\hat{\mu}^{-1}}, 1_G) =  e(e-1). \]
By \cite{Dor}, Theorem 68.1,  $|{\rm IBr}(b(\tilde{G}))| = e < p$. Hence any irreducible Brauer character of $b(\tilde{G})$ is $G$-invariant, hence any $\tilde{\chi}_i$ are $G$-invariant by Lemma 5. On the other hand we see $(1_{G} * \eta_{\hat{\nu}})\downarrow_{\tilde{G}} =  1_{\tilde{G}} * \eta_{{\nu}}$ (cf. Cabanes \cite{Cabanes}, Theorem 1), hence Lemma 6 and the above imply (5). This completes the proof. $\blacksquare$ \end{proof} \\

We call $\chi_{i}$ (resp. ${\chi}_{\nu})$ in the above proposition the $canonical$ $extension$ of $\tilde{\chi}_i$ (resp. $\tilde{\chi}_{\nu})$. We set  $\chi_1 = 1_{G}$ and call it the canonical extension of $\tilde{\chi}_1 = 1_{\tilde{G}}$. Since $\tilde{\chi}_i$ is $p$-rational from (4), $\chi_i$ is also $p$-rational by Proposition 3 and the uniqueness of $\chi_i$.

Let $G_1$ be a subgroup of $G$ containing $\tilde{G}$ and let $P_1 = P \cap G_1 = \tilde{P} \rtimes (C_{P}(T) \cap G_1)$. Then $P_1$ is a Sylow $p$-subgroup of $G_1$,  $\tilde{P}$ is a $p$-hyperfocal subgroup of $G_1$ and $G_1$ also has the inertial index $e$. Let $T$ be a $p$-complement of $N$. Then  $P_1$ is normalized by $T$. From (3) \begin{equation} N_{G_1}(P_1) = O_{p'}(C_{{G_1}}(P_1))TP_1.  \end{equation} That is,  $O_{p'}(C_{G_1}(P_1))T$ is a $p$-complement of $N_{G_1}(P_1)$ because $T \cap C_{G_1}(P_1) = O_{p'}(C_{G}(P))$. Moreover, for $\nu \in {\rm Irr}_{P}(\tilde{P})$,  
\[ 1_{G_1}* \eta_{(\hat{\nu}\downarrow_{P_{1}})} = (1_{G}* \eta_{\hat{\nu}})\downarrow_{G_1}.\] 
Therefore the canonical extension of $\tilde{\chi}_i$ (resp. $\tilde{\chi}_{\nu}$) to $G_1$ coincides with $\chi_{i}\!\! \downarrow_{G_{1}}$ 
(resp. $\chi_{\nu}\!\! \downarrow_{G_{1}}$). 
 
Let $\nu \in $ {\rm Irr}$(\tilde{P})$. We denote by $P_{\nu}$ the stabilizer of $\nu$ in $P$ and set $G_{\nu} = \tilde{G}P_{\nu}$. 
Since $(\eta_{\nu})^v = \eta_{\nu^v}$ and hence $(1_{\tilde{G}}* \eta_{\nu} )^v = 1_{\tilde{G}}* \eta_{\nu^v}$ for $v \in P$,  the equation (5) implies 
\[ (\tilde{\chi}_{\nu})^v = \tilde{\chi}_{\nu^v}. \]
Hence $G_{\nu}$ is the stabilizer of $\tilde{\chi}_{\nu}$ in $G$ because $G = \tilde{G}P$. Since $C_{P}(\tilde{P}) \subseteq P_{\nu}$ and $P/C_{P}(\tilde{P})$ is a cyclic $p$-group isomorphic to a subgroup of Aut$(\tilde{P})$, $P_{\nu}$ is normalized by $N$ and hence $G_{\nu}$ is normal in $G$. Let $\hat{\nu}$ be the canonical extension of $\nu$ to $P_{\nu}$ and  let $\chi_{\nu}$ the canonical extension of $\tilde{\chi}_{\nu}$ to $G_{\nu}$. For $\lambda \in {\rm Irr}(P_{\nu}/\tilde{P}$) and for $t \in N_{p'}$, 
 $[P_{\nu}, t] \subseteq [P, t] \subseteq \tilde{P}$ and hence $\lambda^t = \lambda$. Therefore $(\eta_{\hat{\nu}}\lambda)\!\uparrow^{P}_{P_{\nu}}$ is $N$-invariant, and hence $G$-stable. Moreover we note that Irr$(P_{\nu}/\tilde{P})$ can 
be identified with Irr$(G_{\nu}/\tilde{G})$ through the isomorphism $P_{\nu}/\tilde{P} \cong G_{\nu}/\tilde{G}$.

%Prop 4
\begin{prop} Let $\nu \in $ {\rm Irr}$(\tilde{P})$. With the above notations, 
$ (1_{G_{\nu}} * \eta_{\hat{\nu}})\lambda = 1_{G_{\nu}} * \eta_{\hat{\nu}}\lambda = (e - 1)\lambda - \displaystyle\sum_{i = 2}^{e}\varepsilon_i ({{\chi}_i}\downarrow_{G_{\nu}})\lambda + \ \varepsilon {\chi}_{\nu} \lambda, $ and hence
\begin{equation} 
1_{G} * ((\eta_{\hat{\nu}} \lambda )\!\uparrow^P_{P_{\nu}}) = (e - 1) \lambda\!\uparrow_{G_{\nu}}^{G} -\sum_{i = 2}^{e}\varepsilon_i (({{\chi}_i \downarrow_{G_{\nu}})\lambda})\!\uparrow_{G_{\nu}}^{G} + \ \varepsilon ({{\chi}_{\nu}
\lambda})\!\uparrow_{G_{\nu}}^{G} \end{equation}
for any $\lambda \in ${\rm Irr}$(P_{\nu}/\tilde{P}) = ${\rm Irr}$(G_{\nu}/\tilde{G})$. \end{prop}
\begin{proof} It suffices to show $1_{G} * ((\eta_{\hat{\nu}} \lambda )\!\uparrow^P_{P_{\nu}}) = (1_{G_{\nu}} * (\eta_{\hat{\nu}} \lambda)) \uparrow^G_{G_{\nu}}$. Let $P = \bigcup_{j = 1}^{l} v_j P_{\nu}$ (disjoint) and let $u \in P_{\nu}$. Since $G = \bigcup_{j = 1}^{l}v_j  G_{\nu}$ (disjoint)  and $P_{\nu}$ is normal in $P$, for $\rho \in C_{G}(u)_{p'}$, 
\[ (1_{G_{\nu}} * (\eta_{\hat{\nu}} \lambda))\uparrow^G_{G_{\nu}}\!(u\rho) = \sum_{j = 1}^{l} (1_{G_{\nu}} * (\eta_{\hat{\nu}} \lambda))((u\rho)^{v_j}) \]
\[= \sum_{j = 1}^{l}( \eta_{\hat{\nu}} \lambda)(u^{v_j}) =  
((\eta_{\hat{\nu}} \lambda )\!\uparrow^P_{P_{\nu}})(u) = (1_{G} * ((\eta_{\hat{\nu}} \lambda)\!\uparrow^P_{P_{\nu}}))(u\rho). \]
This completes the proof. $\blacksquare$ 
\end{proof} \\

For $\nu,$ $ \nu' \in \tilde{{\cal X}}$, $\tilde{\chi}_{\nu}$ and 
$\tilde{\chi}_{\nu'}$ are $G$-conjugate if and only if those are $P$-conjugate if and only if $\nu $ and $\nu'$ are $(P\rtimes E)$-conjugate. Since $b(G)$ is the unique block of $G$ which covers $b(\tilde{G})$, we have the following. 

% prop 5
\begin{prop} Let ${\cal X} (\subseteq \tilde{{\cal X}})$ be a set of representatives for the $(P \rtimes E)$-conjugacy classes of {\rm Irr} $(\tilde{P}) \backslash \{ 1_{\tilde{P}} \}$.  
\[ {\rm Irr}(b(G)) = \bigcup_{i = 1}^e \Big\{ \chi_i \lambda \ | \  \lambda \in {\rm Irr}(G/\tilde{G}) \Big\}
\bigcup_{\nu \in {\cal X}} \Big\{ (\chi_{\nu} \lambda)\!\uparrow_{G_{\nu}}^{G} \ | \ \lambda \in {\rm Irr}(G_{\nu}/\tilde{G}) \Big\}. \]
 \end{prop}

\section{Generalized decomposition numbers and a proof of Theorem 1}

 In this section we use notations in Propositions 4 and 5. For $u \in P$, set 
$b_u = b(C_{G}(u))$. We note that if $C_{G}(u)$ is  $p$-nilpotent, then \[ d^{u}_{\chi, 1_{C_{G}(u)}} = \chi(u) \ \ (\forall \chi \in {\rm Irr}(b(G)). \]

%prop 6
\begin{prop} Let $u \in P$ and assume $e_{u} = 1$. We have  
\[ (\chi_{i}\lambda)(u) = \varepsilon_i \lambda(u),  \]
\[ ({\chi_{\nu}}\lambda_{\nu})\!\uparrow_{G_{\nu}}^G\! (u) = \varepsilon(\eta_{\hat{\nu}}\lambda_{\nu})\!\uparrow_{P_{\nu}}^{P}\!(u),  \]
\[ (1 \leq i \leq e, \ \lambda \in {\rm Irr}(P/\tilde{P}) = {\rm Irr}(G/\tilde{G}), \ \nu \in {\cal X},\  \lambda_{\nu} \in {\rm Irr}(P_{\nu}/\tilde{P}) = {\rm Irr}(G_{\nu}/\tilde{G}) ). \]
\end{prop}
\begin{proof} By Lemma 1, $P_u = C_{P}(u)$ is a Sylow $p$-subgroup of $C_{G}(u)$. Here we show that $u$ and $u^{t} $ are not $P$-conjugate for any 
 $u \in P$ with $e_u = 1$ and $t \in T \backslash O_{p'}(C_{G}(P))$ where $T$ is a $p$-complement of $N$. Suppose that $u$ and $u^t$ are $P$-conjugate. Then $\langle t \rangle$ acts on $u^P$ by conjugation. Since $P$ acts on $u^P$ transitively. So by a lemma of Glauberman, there is $u' \in u^P$ such that ${u'}^t = u'$. Hence $u^{v^{-1}tv} = u$ for some $v \in P$. 
 This is a contradiction by Lemma 6. Set $\bar{P} = P/\tilde{P}$ and let $P_{\nu}$ be the stabilizer of $\nu$ in $P$ for $\nu \in {\cal X}$. Moreover we set $\bar{P}_{\nu} = P_{\nu}/\tilde{P}$ and  $x_i = \chi_i(u)$.  We recall $x_i$ is a rational integer. From (7), it suffices to show in order to establish the proposition 
\[x_i = \varepsilon_i \ (1 \leq i \leq e). \]
So we can assume that $P = \langle \tilde{P}, u \rangle$, hence 
\[ G = \langle \tilde{G}, u \rangle. \]
 From (7) again and \cite{NT}, Theorem 5.4.11,  
\[ \sum_{\lambda \in {\rm Irr}(\bar{P})}|\lambda(u)|^2  + \sum_{i = 2}^{e} \sum_{\lambda \in {\rm Irr}(\bar{P})}{x_i}^2|\lambda(u)|^2 \]
\[+ \sum_{\nu}\sum_{\lambda_{\nu} \in {\rm Irr}(\bar{P}_{\nu})}\left|((\eta_{\hat{\nu}}\lambda_{\nu})\!\uparrow_{P_{\nu}}^P\big)(u) - (e -1)(\lambda_{\nu}\!\uparrow^{P}_{P_{\nu}})(u) + \sum_{i = 2}^{e}\varepsilon_{i} x_i (\lambda_{\nu}\!\uparrow^{P}_{P_{\nu}})(u)\right|^2 \] 
\[ = |C_{P}(u)| \]
where $\nu$ runs over a set $\{ \nu \in {\cal X} \ | \ P_{\nu} =  P \}$. 
Hence 
%(8)
\begin{equation}
 1+ \sum_{i = 2}^{e} {x_i}^2 + \sum_{ \nu }\left|\eta_{\hat{\nu}}(u) - (e -1) + \sum_{i = 2}^{e}\varepsilon_{i} x_i \right|^2  = |C_{\tilde{P}}(u)| \end{equation}
where $\nu$ runs over  $\{ \nu \in {\cal X} \ |\  P_{\nu} = P \}$. Set $p^{n'}  = |C_{\tilde{P}}(u)|$. By Brauer's permutation lemma (\cite{NT}, Lemma 3.2.18, (i)), the numbers of  $P$-invariant linear characters of $\tilde{P}$ is equal to $|C_{\tilde{P}}(u)|$. Let ${\cal X}_1 = \{ \nu \in {\cal X} \ | \ \nu \in {\rm Irr}_{P}(\tilde{P}) \}$. We have $|{\cal X}_1| = \frac{p^{n'} - 1}{e}$. Set
 $m = \frac{p^{n'} - 1}{e}$.

By the second orthogonality relation for $P$ and the fact that $u$ and $u^t$ are not $P$-conjugate for any $u \in P$ with $e_u = 1$, and  for any $t \in T \backslash O_{p'}(C_{G}(P))$, we can show $\sum_{\nu \in {\cal X}_1}|\eta_{\hat{\nu}}(u)|^2 = p^{n'} - e$. Moreover by the definition of $\hat{\nu}$ and the fact $u \not\in C_{P}(E)$
 we see $ \sum_{\nu \in {\cal X}_1}{\eta}_{\hat{\nu}}(u) = -1$. In fact by using the second  
 orthogonality relations for $P$  
\[ |P/\tilde{P}| \sum_{\nu \in {\cal X}_{1}}|\eta_{\hat{\nu}}(u)|^2 
= \sum_{t' \in E} \Big( \sum_{t \in E}\sum_{\nu \in {\cal X}_1} \sum_{\lambda \in {\rm Irr}(P/\tilde{P})}{\hat{\nu}}^t (u^{-1}){\hat{\nu}}^t (u^{t'})\lambda(u^{-1})\lambda(u) \]
\[+  \sum_{t \in E}\sum_{\mu \in {\cal X} \backslash {\cal X}_1} \sum_{\lambda_{\mu}}(\hat{\mu}\uparrow^{P}_{P_{\mu}})^{t}(u^{-1})(\hat{\mu}\uparrow^{P}_{P_{\mu}})^{t}(u^{t'})\lambda_{\mu}(u^{-1})\lambda_{\mu}(u) \Big) = 
|C_{P}(u)| - |P/\tilde{P}|e,  \]
where  $\lambda_{\mu}$ runs over ${\rm Irr}(P)/ X_{\mu}$, $X_{\mu}$ is the kernel of the restriction map ${\rm Irr}(P/\tilde{P}) \rightarrow {\rm Irr}(P_{\mu}/\tilde{P})$, because $u$ and $u^{t'}$ are not $P$-conjugate if $t' \neq 1$.  We also note $\lambda(u^{t'}) = \lambda(u)$. Thus we obtain  $\sum_{\nu \in {\cal X}_1}|\eta_{\hat{\nu}}(u)|^2 = p^{n'} - e$. 
Next we show  $ \sum_{\nu \in {\cal X}_1}{\eta}_{\hat{\nu}}(u) = -1$. Write $ u = \tilde{u} r$ 
($\tilde{u} \in \tilde{P}$, $r \in C_{P}(E)$). Then $P = \langle \tilde{P}, r \rangle$. $[\tilde{P}, P] = [\tilde{P}, r] = 
\langle [\tilde{P}, r] \rangle$, $r^P = [\tilde{P}, r]r$. Since $e_u = 1$ and hence $u \not\in r^P$, $\tilde{u} \not\in [\tilde{P}, P]$. 
By the second orthogonality relation for $P/[\tilde{P}, P] \cong (\tilde{P}/[P, \tilde{P}] ) \times C_{P}(E)$, 
\[ 0 = \sum_{t \in E} \sum_{\nu \in {\cal X}_1} \sum_{\lambda \in {\rm Irr}(P/\tilde{P})}
({\hat{\nu}}^t \cdot \lambda)(\tilde{u}) + \sum_{\lambda \in {\rm Irr}(P/\tilde{P})} \lambda(\tilde{u}) \]
\[  = |P/\tilde{P}|\big(\sum_{\nu \in {\cal X}_1}\eta_{\hat{\nu}}(u) +1 \big).  \]
Therefore we have $\sum_{\nu \in {\cal X}_1}\eta_{\hat{\nu}}(u) = -1$. 

\noindent
Now from (8), 
%(9)
\begin{equation}p^{n'}-1-\sum_{i=2}^e{x_i}^2=\sum_{\nu\in\mathcal X_1}\left|\eta_{\widehat\nu}(u)-(e-1)+\sum_{i=2}^e\varepsilon_ix_i\right|^2. \end{equation}
\[{\rm The \ left \ hand \ of} \ (9) = p^{n'}-e-\sum_{i=2}^e\left({x_i}^2-1\right).\]
\[{\rm The \ right\  hand\  of }\ (9) = \sum_{\nu\in\mathcal X_1}\left|\eta_{\widehat\nu}(u)+\sum_{i=2}^e\left(\varepsilon_ix_i-1\right)\right|^2 \]
\[= \sum_{\nu\in\mathcal X_1}\left|\eta_{\widehat\nu}(u)\right|^2+\sum_{\nu\in\mathcal X_1}\left(\eta_{\widehat\nu}(u)+\overline{\eta_{\widehat\nu}(u)}\right)\sum_{i=2}^e(\varepsilon_ix_i-1)+\sum_{\nu\in\mathcal X_1}\left(\sum_{i=2}^e\left(\varepsilon_ix_i-1\right)\right)^2 \]
\[= p^{n'}-e-2\sum_{i=2}^e(\varepsilon_ix_i-1)+m\left(\sum_{i=2}^e\left(\varepsilon_ix_i-1\right)\right)^2. \]
Hence we have 
\begin{eqnarray*}
0 &= & \sum_{i=2}^e\left({x_i}^2-1\right)-2\sum_{i=2}^e(\varepsilon_ix_i-1)+m\left(\sum_{i=2}^e\left(\varepsilon_ix_i-1\right)\right)^2 \\
 &=& \sum_{i=2}^e\left(\left({x_i}^2-1\right)-2(\varepsilon_ix_i-1)\right)+m\left(\sum_{i=2}^e\left(\varepsilon_ix_i-1\right)\right)^2 \\ 
 &=& \sum_{i=2}^e\left({x_i}^2-2\varepsilon_ix_i+1\right)+m\left(\sum_{i=2}^e\left(\varepsilon_ix_i-1\right)\right)^2 \\
&=& \sum_{i=2}^e(\varepsilon_ix_i-1)^2+m\left(\sum_{i=2}^e\left(\varepsilon_ix_i-1\right)\right)^2. 
\end{eqnarray*}
Hence we have 
\[\varepsilon_ix_i-1=0 \therefore x_i = \varepsilon_i \ ( 2 \leq i \leq e). \]
This completes the proof. $\blacksquare$ \end{proof} \\

Since $|{\rm IBr}(b(\tilde{G}))| = e < p$ and $G/\tilde{G}$ is a $p$-group, 
$|{\rm IBr}(b({G}))| = e$. In particular if $e_u = e$ where $u \in P$, then $l(b_u) = e$. 

% prop 7
\begin{prop} Let $u \in P$ and assume $e_{u} = e$. The matrix of generalized decomposition numbers of $b(G)$ with respect to a suitable basic set  $\{ \varphi_1^{(u)} = 1_{C_{G}(u)}, \varphi_2^{(u)}, \cdots, \varphi_e^{(u)} \}$ for  $b_u$ 
is of the form : 
  \[ \left.\begin{array}{|c|cccc|c}
\hline
 & \varphi_1^{(u)} = 1_{C_{G}(u)}&\varphi_2^{(u)} &\cdots &\varphi_e^{(u)} \\
\hline
\chi_{1}\lambda_1&\varepsilon_{1}\lambda_1(u)& 0&\cdots &0  \\
\chi_{2}\lambda_2&0&\varepsilon_{2}\lambda_2(u)& 0& 0\\
 \vdots&\vdots &\vdots &\ddots & \vdots  \\
\chi_{e}\lambda_e& 0& 0&\cdots&\ \varepsilon_e \lambda_e(u) \\ 
(\chi_{\nu }\lambda_{\nu})\!\uparrow^G_{G_{\nu}}&
\varepsilon ({\lambda_{\nu}}\!\uparrow^{P}_{P_{\nu}})(u)&\varepsilon ({\lambda_{\nu}}\!\uparrow_{P_{\nu}}^{P})(u) &\cdots &\varepsilon ({\lambda_{\nu}}\!\uparrow^{P}_{P_{\nu}})(u) \\
\hline
\end{array} \right. \]
where $\lambda_i \in {\rm Irr}(P/\tilde{P}) = {\rm Irr}(G/\tilde{G})$, $\nu \in {\cal X}$ and $\lambda_{\nu} \in {\rm Irr}(P_{\nu}/\tilde{P}) = {\rm Irr}(G_{\nu}/\tilde{G})$.  
\end{prop}  

\begin{proof} By Lemmas 2 and 3, $P_u = C_{P}(u)$ is a Sylow $p$-subgroup of $C_{G}(u)$, $u \in C_{P}(T)$ for some $p$-complement $T$ of $N$  and $C_{\tilde{P}}(u)$ is a $p$-hyperfocal subgroup of $C_{G}(u)$. Moreover $l(b_u) = e$ as is stated in the above. Set $p^{n'} = |C_{\tilde{P}}(u)|$ and $m = \frac{p^{n'} -1}{e}$. Now for $\nu \in {\cal X}$, the stabilizer $P_{\nu}$ of $\nu$ in $P$ is normalized by $N$. If $u \not\in  P_{\nu}$, $\eta_{\nu}\!\!\uparrow^P_{P_{\nu}}\!\!(u) = 0$ and $\chi_{\nu}\!\!\uparrow^G_{G_{\nu}}\!\!(u \rho) = 0$ for any $\rho \in C_{G}(u)_{p'}$. Recall that if $u \in P_{\nu}$, then $\eta_{\hat{\nu}}(u) = e$ (see (6) for $G_1 = G_{\nu}$). So from (7) in Proposition 4, in order to establish the proposition, it suffices to show for any $\chi_i$ 
\begin{equation}
d_{{\chi_i, \varphi_{j}^{(u)}}}^u = \delta_{ij}\varepsilon_i 
\end{equation}
for a suitable basic set $\{ \varphi_{j}^{(u)} \ | \ 1\leq j \leq e \}$ for $b_u$, where $\varphi_1 ^{(u)} = 1_{C_{G}(u)}$.  Hence we may assume 
\[G = \langle \tilde{G}, u \rangle  = \tilde{G} \rtimes \langle u \rangle. \]
(cf. Lemma 2) Set $\bar{C}_{G}(u) = C_{G}(u)/\langle u \rangle = 
(C_{\tilde{G}}(u) \times \langle u \rangle )/\langle u \rangle \cong C_{\tilde{G}}(u)$, $\bar{b}_u = b(\bar{C}_{G}(u))$ and 
$\overline{C_{P}(u)}  = C_{P}(u)/\langle u \rangle  \cong C_{\tilde{P}}(u)$. Then $\bar{b}_u$ has a cyclic defect group $\bar{P}$, and 
$\bar{b}_u$ has an inertial index $e$. By \cite{Broue}, Theorem 5.3, there is a perfect isometry from $\bar{b}_u$ to $b(N_{\bar{C}_{{G}}(u)}(\bar{P}))$ which maps $1_{\bar{C}_{G}(u)}$ to $1_{N_{\bar{C}_{G}(u)}(\bar{P})}$. Hence by \cite{Broue}, Theorem 1.5, there exists a basic set $Bs(\bar{b}_u) = \{ \varphi_{j}^{(u)}\ | \ 1\leq j \leq e \}$, where $\varphi_{1}^{(u)} = 1_{\bar{C}_{G}(u)}$, such that the Cartan matrix of $\bar{b}_u$ with respect to $Bs(\bar{b}_u)$ is of the form 
\[ {C}^{\bar{b}_u} = \left( \begin{array}{cccc}
{m + 1}&m & \cdots & m\\
m & m+1 &  \cdots &m \\
\vdots&\vdots&\vdots&\vdots\\
m &m & \cdots& m + 1
\end{array} \right)_{e \times e}. \] 
Let $C^{b_u}$ be the Cartan matrix of $b_u$ with respect to $Bs(\bar{b}_u)$ which is regarded as a basic set for $b_u$. We have $C^{b_u} = |\langle u \rangle |C^{\bar{b}_u}$ by \cite{NT}, Theorem 5.8.11. Let 
\[D = \Big( {d}_{\chi {\varphi}_{j}^{(u)}}^{u}\Big)_
{\chi \in {\rm Irr}(b(G)), \varphi_{j}^{(u)} \in Bs(\bar{b}_u)} \]
be the matrix of generalized decomposition numbers of $b(G)$ with respect to $Bs(\bar{b}_u)$. Then we have 
\begin{equation}
C^{b_u} = {^{t}D\overline{D}}. 
\end{equation}

As in the proof of Proposition 6, let ${\cal X}_1 = \{ \nu \in {\cal X} \ | \ \nu \in {\rm Irr}_{P}(\tilde{P}) \}$. We have $|{\cal X}_1| = m$. Let $\nu \in {\cal X}_1$. Recalling  $\eta_{\hat{\nu}}(u) = e$, from (7), 
\[ed_{\chi_1 \varphi_j^{(u)}}^u = (e - 1) d_{\chi_1 \varphi_j^{(u)}}^u - \sum_{k = 2}^{e} \varepsilon_k d_{\chi_k \varphi_j^{(u)}}^u + \varepsilon d_{\chi_{\nu} \varphi_{j}^{(u)}}^{u}, \]
\[\therefore \ \ \ \varepsilon d_{\chi_{\nu} \varphi_j^{(u)}}^u = \sum_{k = 1}^{e}\varepsilon_k d_{\chi_k \varphi_j^{(u)}}^u  \]
for $j = 1, 2, \cdots, e$. Note that these are all integers. Since $0 \not\equiv \chi_{\nu}(1) \equiv \chi_{\nu}(u) \ ({\rm mod} \ {\cal P})$ where 
${\cal P}$ is the maximal ideal of ${\cal O}$, $\chi_{\nu}(u) \neq 0$. In particular $d^{u}_{\chi_{\nu}\varphi_{j}^{(u)}} \neq 0$ for some $j$. The same holds for $\chi_i$. Now set 
\[ X_{kj}  = \varepsilon_k d_{\chi_k \varphi_j^{(u)}}^u \ ( 1 \leq k, j \leq e).   \] 
Then $X_{11} = 1$, $X_{1j} = 0 \ \ (2 \leq j \leq e)$. Moreover $\varepsilon d_{\chi_{\nu} \varphi_j^{(u)}}^u = \sum_{k = 1}^{e}X_{kj}$ for $\nu \in {\cal X}_1$. So we can see from (11), 
\[  \sum_{k = 1}^{e}X_{ki}X_{kj} + m\big(\sum_{k= 1}^{e}X_{ki}\big) \big(\sum_{k= 1}^{e}X_{kj }\big) = \left\{ \begin{array}{c}m + 1 \ (i = j) \\
m \  (i \neq j).  
\end{array} \right. \]
Therefore if $m \geq 2$, then $X_{ij} = \delta_{ij}$ by rearranging  $\varphi^{(u)}_j$ $(j \geq 2)$. But if $m = 1$, by taking another suitable basic set 
$\{ {\varphi'}_{j}^{(u)} \ | \ 1\leq j \leq e \}$ for $b_u$ where ${\varphi'}_{1}^{(u)} = 1_{C_{G}(u)}$ if necessary, 
$\varepsilon_id_{\chi_i {\varphi'}_j^{(u)}}^u = \delta_{ij}$ with respect to
$\{ {\varphi'}_{j}^{(u)} \ | \ 1\leq j \leq e \}$. 
 Thus we get (10). This completes the proof. $\blacksquare$ \end{proof}\\

{\bf Proof of Theorem 1.} It is easily seen that $\tilde{P}$ is a hyperfocal subgroup of $N$ and $N_{G}(\tilde{P})$ by Proposition 2 and Proposition 1. Let $\tilde{N} = O^{p}(N) = N \cap \tilde {G}$ and 
$N_{\nu} = N \cap G_{\nu}$ for $\nu \in {\cal X}$. Then Propositions 3-7 hold for $b(N)$. In fact there exist $e$ irreducible characters $\chi'_i \in b(N)$ and  $\chi'_{\nu}$ of $b(N_{\nu})$ corresponding to $\nu \in {\cal X}$  such that 
for $\lambda_{\nu} \in ${\rm Irr}$(P_{\nu}/\tilde{P}) = ${\rm Irr}$(N_{\nu}/\tilde{N})$, 
\[1_{N_{\nu}} * (\eta_{\hat{\nu}}\lambda_{\nu}) = (e - 1)\lambda_{\nu} - 
\sum_{i = 2}^{e} ({\chi'_i}\!\downarrow_{N_{\nu}}\lambda_{\nu}) + \ \chi'_{\nu}\lambda_{\nu}, \]
and hence 
\[
1_{N} * ((\eta_{\hat{\nu}} \lambda_{\nu})\uparrow^P_{P_{\nu}}) = (e - 1)(\lambda_{\nu}\!\uparrow_{N_{\nu}}^{N}) - \sum_{i = 2}^{e}({\chi'_i \!\downarrow_{N_{\nu}}\!\lambda_{\nu}})\!\uparrow_{N_{\nu}}^{N} + \ ({\chi'_{\nu}\lambda_{\nu}})\!\uparrow_{N_{\nu}}^{N}. \]  
By Proposition 5 for $N$, 
 ${\rm Irr}(b(N)) = \bigcup_{i = 1}^e \Big\{ \chi'_i \lambda \ | \  \lambda \in {\rm Irr}(N/\tilde{N}) \Big\} \bigcup_{\nu \in {\cal X}} \Big\{ (\chi'_{\nu} \lambda_{\nu})\!\uparrow_{N_{\nu}}^{N} \ | \ \lambda_{\nu} \in {\rm Irr}(N_{\nu}/\tilde{N}) \Big\}$ and by Theorem 2, 
\[ \left\{ \begin{array}{l}
\chi'_i \lambda \mapsto \varepsilon_{i} \chi_i \lambda \\
(\chi'_{\nu}\lambda_{\nu})\uparrow^{N}_{N_{\nu}} \mapsto \varepsilon (\chi'_{\nu}\lambda_{\nu})\!\uparrow^{G}_{G_{\nu}}
\end{array} \right. \]
gives a perfect isometry from $b(N)$ to $b(G)$. Similarly $b(G)$ and $b(N_{G}(\tilde{P}))$ also are perfect isometric. $\blacksquare$ \\

\vspace{20pt}
{\small
{\sc National Institute of Technology, Kumamoto College, 2659-2, Suya, Koushi, Kumamoto 861-1102, Japan}

{\it E-mail} : {\tt hori@kumamoto-nct.ac.jp}

\vspace{10pt}

{\sc Kumamoto University, 1-16-37, Toroku, Chuo-ku, Kumamoto 862-0970, Japan}

{\it E-mail} : {\tt qqsh2tn9kq@io.ocn.ne.jp}
}

\end{document}